\documentclass[12pt]{article}
\usepackage[utf8]{inputenc}
\usepackage{amsthm}
\usepackage{amsmath}
\usepackage{mathtools}
\usepackage{bbm}
\usepackage{amsfonts}
\usepackage{amssymb}
\usepackage[left=2.5cm,right=2.5cm,top=2cm,bottom=2cm]{geometry}

\usepackage[dvipsnames]{xcolor}

\usepackage{tabto}
\TabPositions{2 cm, 4 cm, 6 cm, 8 cm}

\usepackage{tikz-cd}

\usepackage{tikz}
\usetikzlibrary{arrows, automata, positioning}

\usepackage{tocloft}
\usepackage{appendix}

\usepackage{enumerate}

\usepackage{hyperref}
 \newcommand{\Manoa}{M\=anoa }
\newcommand{\Hawaii}{Hawai\kern.05em`\kern.05em\relax i }
\newtheorem{theorem}{Theorem}[section]
\newtheorem{corollary}[theorem]{Corollary}
\newtheorem{lemma}[theorem]{Lemma}

\newtheorem{proposition}[theorem]{Proposition}

\theoremstyle{definition}
\newtheorem{definition}[theorem]{Definition}
\newtheorem{example}[theorem]{Example}
\newtheorem*{definition*}{Definition}

\theoremstyle{remark}

\newtheorem*{lemma*}{Lemma}
\newtheorem*{proposition*}{Proposition}
\newtheorem*{theorem*}{Theorem}
\newtheorem*{corollary*}{Corollary}
\newtheorem*{claim*}{Claim}

\theoremstyle{definition}

\newtheorem{question}[theorem]{Question}

\newcommand{\Homeo}{\operatorname{Homeo}}

\newcommand{\rot}{\operatorname{rot}}

\begin{document}

\title{Finitely presented left orderable monsters}
\author{Francesco Fournier-Facio, Yash Lodha and Matthew C. B. Zaremsky}
\date{\today}
\maketitle

\begin{abstract}
A left orderable monster is a finitely generated left orderable group all of whose fixed point-free actions on the line are \emph{proximal}: the action is semiconjugate to a minimal action so that for every bounded interval $I$ and open interval $J$, there is a group element that sends $I$ into $J$. In his 2018 ICM address, Navas asked about the existence of left orderable monsters. By now there are several examples, all of which are finitely generated but not finitely presentable. We provide the first examples of left orderable monsters that are finitely presentable, and even of type $F_\infty$. These groups satisfy several additional properties separating them from the previous examples: they are not simple, they act minimally on the circle, and they have an infinite-dimensional space of homogeneous quasimorphisms. Our construction is flexible enough that it produces infinitely many isomorphism classes of finitely presented (and type $F_{\infty}$) left orderable monsters.
\end{abstract}

\section{Introduction}

A group $G$ is \emph{left orderable} if it admits a total order which is invariant under left multiplication by any group element.
This concept has a remarkable connection with dynamics of group actions on the line: a countable group $G$ is left orderable if and only if it admits a faithful action by orientation-preserving homeomorphisms on the real line. 

One of the fundamental questions in the field is to understand the following. Given a left orderable group, what are the possible actions on the line by orientation-preserving homeomorphisms?  It is folklore that up to semiconjugacy, every action of a finitely generated left orderable group on the line by orientation-preserving homeomorphisms and without global fixed points is of one of three types \cite[Theorem 3.5.19]{book}:

\begin{enumerate}
    \item \emph{Cyclic} or \emph{type 1}, if there exists an invariant Radon measure on $\mathbf{R}$, in which case the group surjects onto the integers.
    \item \emph{Locally proximal} or \emph{type 2}, if it does not preserve a Radon measure on $\mathbf{R}$, yet it is semiconjugate to a minimal action that commutes with integer translations. Equivalently, such an action descends to a minimal and proximal action on the circle (see Proposition \ref{prop:type2:prox}).
    \item \emph{Proximal} or \emph{type 3}, if for every bounded interval $I$ and open interval $J$, there is a group element $f$ such that $I\cdot f\subset J$.
\end{enumerate}

The above framework has shed light on various algebraic and dynamical properties of left orderable groups, and on the question of when certain groups are left orderable.
For example, it was unknown until very recently whether lattices in higher-rank semisimple Lie groups admit actions on the line or not; it turns out that outside some exceptional cases they do not \cite{deroin:hurtado}. Before this breakthrough solution, given that lattices do not admit actions of type 1 or 2 \cite{lattices:ghys, lattices:bm, rigidity} it was thought that a promising approach could be to leverage finite generation to also rule out type 3 actions. More succinctly, \cite[Question 4]{questions} (see also \cite[Question 3.5.11]{book}) asks:

\begin{question}[Navas ICM $2018$]
Do there exist finitely generated left orderable groups all of whose fixed point-free actions by homeomorphisms on the line are of type 3? 
\end{question}

\begin{definition}[Navas]
A finitely generated left orderable group $G$ is a \emph{left orderable monster} if all of its fixed point-free actions by homeomorphisms on the line are of type 3.
\end{definition}

Given the negative solution in \cite{deroin:hurtado} to the question concerning higher-rank lattices, it is intriguing that in fact left orderable monsters \emph{do} exist. This fact was proved independently by Matte Bon and Triestino \cite{Tphi}, and by Hyde, the second author, Navas and Rivas \cite{Grho, Grho:up}. In both cases the examples are families of groups whose construction is non-trivial, and that have a number of stronger properties. First, both families of groups are simple: a property much stronger than the absence of type 1 actions. Secondly, both groups are uniformly perfect, and even have vanishing second bounded cohomology \cite{spectrum}; in fact they admit no fixed point-free action on the circle whatsoever \cite{Tphi, spectrum} -- these properties are much stronger than the absence of type 2 actions. 

However, the groups in \cite{Tphi,Grho, Grho:up} are \emph{not} finitely presentable since each is a non-trivial limit of markings of a finitely presented group in the space of marked groups. Therefore one might hope that the above still yields an alternative approach to disprove left orderability of lattices, namely leveraging finite presentability to rule out type 3 actions. In particular, one may hope that this approach works for lattices in products of trees \cite{burger:mozes}, whose left orderability is still an open question. Indeed, it follows from \cite[Corollary 26]{rigidity} that if they are left orderable, then they must be left orderable monsters. One of the key motivations of this article is to show that this modified approach cannot work, that is, there do in fact exist finitely presentable left orderable monsters.

Our examples are very different in nature from the ones in \cite{Grho, Tphi}. First, as we already mentioned they are finitely presentable, and even of type $F_\infty$ (finite presentability is equivalent to admitting a classifying space with a finite $2$-skeleton, and type $F_\infty$ means there is a classifying space with a finite $n$-skeleton for all $n$). Secondly, the absence of type 1 or type 2 actions is shown via a direct dynamical argument, and not via some much stronger property such as simplicity, uniform perfection or vanishing of second bounded cohomology. Indeed, each of these properties is shown to fail. Lastly, in our construction the group itself is much more straightforward than in previous examples.

We now describe a concrete instance of our general construction.
Let $\overline{T}$ denote the lift of Thompson's group $T$ to the real line. Let $z$ denote the generator of the center of $\overline{T}$.
Let $\overline{T}_{2,3}$ denote the lift of the Stein-Thompson group $T_{2,3}$, and let $g\in \overline{T}_{2,3}$ denote a lift of an element with an irrational rotation number in $T_{2,3}$. Such elements exist, e.g., it was proved in \cite{liousse} that $T_{2,3}$ contains elements whose rotation number equals $\frac{\log 2}{\log 3}$.
 We now take the amalgamated product and define:
$$ M := \overline{T} *_{z = g} \overline{T}_{2,3}.$$
Now we can state our main result:

\begin{theorem}
\label{thm:main23}
The group $M$ is a left orderable monster.
\end{theorem}

A number of notable properties can be quickly established for this group. 
Firstly, the space of homogeneous quasimorphisms of $M$ is infinite-dimensional. In particular, $M$ is not uniformly perfect, and the second bounded cohomology of $M$ is infinite-dimensional.
Moreover, $M$ surjects onto Thompson's group $T$. In particular, there exists a minimal action of $M$ on the circle, and $M$ is not simple.
This showcases the difference with previous examples, which not only do not share these properties, but also were all more complicated to construct and harder to analyze.

\medskip

To state our more general construction, we define $\mathcal{G}$ to be the class of infinite groups $G$ of orientation-preserving homeomorphisms of the circle that satisfy:
\begin{enumerate}

\item The lift $\overline{G}$ of $G$ to the real line is perfect.

\item $G$ is simple.

\item $G$ has a unique proximal action on the circle, up to conjugacy.

\item $G$ has torsion.

\end{enumerate}

Groups in $\mathcal{G}$ include: Thompson's group $T$ \cite{cfp}, the Stein--Thompson groups $T_{n_1,\dots,n_k}$, whenever $n_2=n_1^{2a}-1$ for some $a$ \cite{stein,guelman} (for example $T_{2,3}$), the commutator subgroup of the Golden Ratio Thompson group $T_{\tau}$ \cite{Ttau}, and the piecewise projective group $S$ constructed by the second author \cite{S}. Some of these groups are known to have irrational rotation numbers, e.g., $T_\tau'$ \cite{spectrum}, and $T_{n_1,\dots,n_k}$ if some $n_i$ and $n_j$ are coprime \cite{liousse}. A peculiarity of $T$ is that the rotation numbers of all of its elements are rational \cite{GhysSergiescu}. We do not know whether $S$ contains an element with irrational rotation number.

The following is our main result (with Theorem~\ref{thm:main23} as a special case).

\begin{theorem}
\label{thm:main}
Let $G_1,G_2\in \mathcal{G}$ be groups such that $G_2$ contains an element with an irrational rotation number. Let $\overline{G}_i$ denote the lift of the standard action of $G_i$, let $z$ denote the generator of the center of $G_1$, and let $g \in \overline{G}_2$ denote the lift of an element with irrational rotation number. Then the following hold:
\begin{enumerate}
\item The group $M := \overline{G}_1*_{z = g}\overline{G}_2$ is a left orderable monster.

\item If $G_1,G_2$ are finitely presented (respectively of type $F_{\infty}$), then $M$ is also finitely presented (respectively of type $F_{\infty}$.)
\item The space of homogeneous quasimorphisms of $M$ is infinite-dimensional. In particular, $M$ is not uniformly perfect, and the second bounded cohomology of M is infinite-dimensional.

\item $M$ admits a unique proximal action on the circle (up to conjugacy), and this coincides with the unique proximal action of the quotient $G_1$.
In particular, $G_1$ is an isomorphism invariant of $M$.

\end{enumerate}

\end{theorem}

Using the above, we will conclude the following.

\begin{corollary}
\label{cor:main}
There exist infinitely many isomorphism classes of finitely presentable (and type $F_{\infty}$) left orderable monsters.
\end{corollary}

\textbf{Acknowledgments.} The first author was supported by an ETH Z\"urich Doc.Mobility Fellowship. The third author was supported by grant \#635763 from the Simons Foundation.

\section{Actions on the line and on the circle}

We denote by $\Homeo^+(\mathbf{R})$ and $\Homeo^+(\mathbf{S}^1)$ the groups of orientation-preserving homeomorphisms of the real line and the circle, respectively. Given $f\in \Homeo^+(\mathbf{R})$, the support of $f$, denoted $supp(f)$, is the set of points that are moved by $f$.
Note that the support of a homeomorphism is always an open set.
Given $G \leq \Homeo^+(\mathbf{R})$ or $\Homeo^+(\mathbf{S}^1)$, we denote by $Fix(G)$ the set of global fixed points of $G$, which is closed. We say that an action $\rho$ of a group $G$ is \emph{fixed point-free} if $Fix(\rho(G)) = \varnothing$.

Recall that a group action on a topological space by homeomorphisms is \emph{minimal} if each orbit is dense.
We say that $\Lambda\subset \mathbf{R}$ (or $\Lambda\subset \mathbf{S}^1$) is \emph{exceptional} if $\Lambda$ is
perfect and totally disconnected. 
Let $G\leq \Homeo^+(\mathbf{M})$, where $\mathbf{M}\in \{\mathbf{R},\mathbf{S}^1\}$, and let $\Lambda\subset \mathbf{M}$ be an exceptional set. We say that $\Lambda$ is \emph{minimal} if it is $G$-invariant and the action of $G$ on $\Lambda$ is minimal.

The following are some fundamental results concerning actions of groups on the circle and the real line:

\begin{theorem}[{\cite[Theorem 2.1.1]{Navas}}]
\label{thm:circleAction}
If $\Gamma\leq \Homeo^+(\mathbf{S}^1)$ is a countable group, then precisely one of the following holds:
\begin{enumerate}

\item The action admits a finite orbit.

\item The action is minimal.

\item There exists a unique, minimal, exceptional subset $\Lambda\subset \mathbf{S}^1$ which is homeomorphic to the
Cantor set.

\end{enumerate}
\end{theorem}

For the case of the line, a similar result holds under the additional assumption of finite generation:

\begin{proposition}[{\cite[Proposition 2.1.12]{Navas}}]
\label{prop:lineAction}
Every fixed point-free action of a finitely generated group $G$ on $\mathbf{R}$ by orientation-preserving homeomorphisms admits a non-empty minimal invariant closed set. This closed set is either:
\begin{enumerate}

\item Discrete, in which case the group surjects onto the integers.

\item The whole real line, in which case the group action is minimal.

\item A unique minimal, exceptional set $\Lambda\subset \mathbf{R}$.

\end{enumerate}
\end{proposition}

We are particularly interested in the third case of the previous proposition. The following is a fact concerning this case:

\begin{proposition}[{\cite[Sections 2.1.1 and 2.1.2]{Navas}}]
\label{prop:minimalization}
Let $\mathbf{M}\in \{\mathbf{R},\mathbf{S}^1\}$.
Let $G\leq \Homeo^+(\mathbf{M})$ be such that $G$ admits a unique, exceptional,
minimal set $\Lambda\subset \mathbf{M}$. Then there is a continuous surjective map $\phi:\mathbf{M}\to \mathbf{M}$, a group
$H\leq \Homeo^+(\mathbf{M})$ and a surjective homomorphism $\psi:G\to H$, such that:
\begin{enumerate}
\item The restriction of $\phi$ to $\Lambda$ is surjective, and every fiber has size at most $2$.
\item For each $g\in G$, $\phi\circ g=\psi(g)\circ \phi$.
\item The action of $H$ on $\mathbf{M}$ is minimal.
\end{enumerate}
\end{proposition}

Using the above, we define the following process, which we call the \emph{minimalization} of an action. We start with the action of a finitely generated group $G$ on $\mathbf{R}$ (or $\mathbf{S}^1$) by orientation-preserving homeomorphisms.
We assume that this action is not minimal; in the case of $\mathbf{R}$ we assume that there is no discrete orbit, and in the case of $\mathbf{S}^1$ we assume that there is no finite orbit. Using the above, we obtain the (possibly non-faithful) action of $G$ on $\mathbf{R}$ (or $\mathbf{S}^1$) which is minimal. This new action is called the \emph{minimalization} of the original action, and is semiconjugate to the original. If the original action was minimal to begin with, then this is the same as the original action.

For actions on the circle, a further property will be relevant for us. An action of a group $G$ on $\mathbf{S}^1$ is said to be \emph{proximal} if for every non-empty open arc $J \subset \mathbf{S}^1$ and every proper arc $I \subset \mathbf{S}^1$ there exists $f \in G$ such that $I \cdot f \subset J$. We warn the reader that in the literature this property is referred to by various other names, such as strongly proximal, extremely proximal, strongly expansive, or compact-open transitive (some of these properties have well-established and distinct meanings in topological dynamics, that happen to coincide for the case of actions on the circle). In this article, we refer to it as proximal.

\begin{theorem}[Margulis \cite{margulis}, see also {\cite[Section 5.2]{ghys:book}}]
\label{thm:prox}

Let $G$ act faithfully and minimally on the circle. Suppose that $G$ is non-abelian. Then:
\begin{enumerate}
    \item If the action is proximal, the centralizer of $G$ in $\Homeo^+(\mathbf{S}^1)$ is trivial.
    \item Otherwise, the centralizer of $G$ in $\Homeo^+(\mathbf{S}^1)$ is a finite cyclic group $\langle \theta \rangle$, and the induced action of $G$ on $\mathbf{S}^1 / \langle \theta \rangle \cong \mathbf{S}^1$ is proximal.
\end{enumerate}
\end{theorem}

Now we can push our minimalization process one step further. Given a minimal action of $G$ on the circle such that the image of $G$ in $\Homeo^+(\mathbf{S}^1)$ is non-abelian, we obtain a proximal action of $G$ on the circle, which we call the \emph{proximalization} of the original action. If the original action was proximal to begin with, then this is the same as the original action. \\

Proximal actions on the circle are especially relevant for type $2$ actions on the line:

\begin{proposition}[{\cite[Theorem 1]{malyutin}}]
\label{prop:type2:prox}

Let $G$ be a group, and let $\rho$ be a minimal type $2$ action on $\mathbf{R}$ by orientation-preserving homeomorphisms. Then there exists a fixed point-free homeomorphism $t \in \Homeo^+(\mathbf{R})$ which commutes with $\rho(G)$, such that the induced action of $G$ on $\mathbf{R}/\langle t \rangle \cong \mathbf{S}^1$ is proximal.
\end{proposition}

Finally, we prove a basic fact about rotation numbers of homeomorphisms of the circle, which will be the crucial tool for our proof of Theorem \ref{thm:main}. Let $g \in \Homeo^+(\mathbf{S}^1)$, and let $\overline{g}$ be a lift of $g$ to the real line. Then the quantity
$$\rot(\overline{g}) := \lim\limits_{n \to \infty} \frac{\overline{g}^n(x)}{n} \in \mathbf{R}$$
is independent of the choice of $x \in \mathbf{R}$, and moreover depends on the choice of lift only up to integers. In particular, the \emph{rotation number} $\rot(g) \in \mathbf{R}/\mathbf{Z}$ is well-defined.

\begin{proposition}
\label{prop:rot:general}

Let $\rho : G \to \Homeo^+(\mathbf{S}^1)$ be a group action on the circle. Let $\rho_m$ be the minimalization of $\rho$, and suppose that $\rho_m(G)$ is non-abelian, so that there exists a proximalization $\rho_p$. Then there exists an integer $n \geq 1$ such that for all $g \in G$ we have $\rot(\rho_p(g)) = n \rot(\rho(g))$.
In particular, the rationality or irrationality of $\rot(\rho(g))$ is preserved under proximalization.
\end{proposition}

\begin{proof}
The rotation number of a circle homeomorphism is a semiconjugacy invariant \cite[Section 5.1]{ghys:book}. In particular, if $\rho_m$ is the minimalization of $\rho$, then for all $g \in G$ it holds that $\rot(\rho_m(g)) = \rot(\rho(g))$.

Now by Proposition \ref{thm:prox}, since $\rho_m(G)$ is non-abelian there exists a (possibly trivial) finite-order homeomorphism $\theta$ that centralizes $\rho_m(G)$ and such that the proximalization $\rho_p$ is the induced action on the quotient $\mathbf{S}^1 / \langle \theta \rangle$. Let $n \geq 1$ be the order of $\theta$. By \cite[Proposition 4.1]{ghys:book}, up to conjugacy and changing the generator of $\langle \theta \rangle$, we may assume that $\theta$ is rotation by $1/n$. It then follows from the definition of rotation numbers that $\rot(\rho_m(g)) = n \rot(\rho_p(g))$, and we conclude.
\end{proof}

\section{Groups in $\mathcal{G}$ and their lifts}\label{sec:properties}
In this section we state and prove various results about the groups $G\in \mathcal{G}$ that play a crucial role in our general construction.
Given a group $G \in \mathcal{G}$, recall that $\overline{G}$ denotes the \emph{lift} of $G$ to the real line, that is the group of all homeomorphisms of $\mathbf{R}$ that commute with integer translations and induce elements of $G$ on the quotient $\mathbf{R}/\mathbf{Z}=\mathbf{S}^1$. We have a central extension
$$1 \to \mathbf{Z} \to \overline{G} \to G \to 1.$$

\begin{lemma}
\label{lem:normalgeneration}
Let $G\in \mathcal{G}$ and let $f\in \overline{G}$ be an element that does not lie in the center. Then $f$ normally generates $\overline{G}$.
\end{lemma}

\begin{proof}
Let $N=\langle\!\langle f\rangle\!\rangle$ and let $Z=Z(\overline{G})\cong\mathbf{Z}$. The image of $N$ under the quotient $\overline{G}\to G$ is a normal subgroup of $G$, and is non-trivial since $N$ is not contained in $Z$. Since $G$ is simple, this image must equal $G$. This implies $\overline{G}=NZ$, so $\overline{G}/N\cong Z/(N\cap Z)$ is abelian. Now the fact that $\overline{G}$ is perfect implies $N=\overline{G}$.
\end{proof}

\begin{lemma}
\label{lem:TbarQuotients}
For $G\in \mathcal{G}$, every proper normal subgroup of $\overline{G}$ is central. Hence $\overline{G}$ has no proper non-trivial torsion-free quotients, and no non-trivial finite quotients.
\end{lemma}

\begin{proof}
That every proper normal subgroup is central follows from Lemma~\ref{lem:normalgeneration}. The center itself corresponds to the quotient $G$, which has torsion. For any other non-trivial proper normal subgroup, the generator of the center will map to a non-trivial torsion element in the quotient. All non-trivial quotients surject onto $G$, and thus are infinite.
\end{proof}

Next we want to study the possible dynamics of $\overline{G}$ ($G\in \mathcal{G}$) on the line (see \cite[Theorem 8.7]{Tphi} for the case of Thompson's group $T$):

\begin{lemma}
\label{lem:lineactions}
For $G\in \mathcal{G}$, every fixed point-free action of $\overline{G}$ by homeomorphisms on the line is semiconjugate to the standard action.
\end{lemma}

\begin{proof}
Let $\rho$ be a fixed point-free action of $\overline{G}$ by orientation-preserving homeomorphisms on the line. Note that $\rho$ cannot fix a Radon measure because $\overline{G}$ is perfect, hence non-indicable. Therefore via a semiconjugacy, we pass to a minimal action (Proposition \ref{prop:minimalization}), which we denote as $\rho'$.
We also know from Lemma \ref{lem:TbarQuotients} that $\rho'$ must be a faithful action of $\overline{G}$.
Since $\overline{G}$ has a cofinal central element, it cannot admit a minimal faithful type $3$ action \cite[Proposition 3.5.12]{book}.
Therefore, $\rho'$ is a type $2$ action. By Proposition \ref{prop:type2:prox}, there exists a fixed point-free homeomorphism $t$ that commutes with $\rho'(G)$ and such that the induced action of $G$ on $\mathbf{R}/ \langle t \rangle \cong \mathbf{S}^1$ is proximal.

Now using our rigidity assumption on proximal actions of groups in $\mathcal{G}$, this action is even topologically conjugate to the standard action of $G$. Lifting this conjugacy, we obtain that $\rho'$ is conjugate to an action $\rho''$ of $\overline{G''}$ on the line that sends the center to itself and induces the standard representation on the quotient $G$. In other words, we can write $\rho''(g) = \rho_{st}(g) z^{i_g}$, where $z$ is the generator of the center, and $i_g \in \mathbf{Z}$. Since $z$ is central, the map $g \mapsto i_g$ is a homomorphism, which must be trivial because $\overline{G}$ is perfect. Therefore $\rho'' = \rho_{st}$ and we conclude.
\end{proof}

Finally we consider actions of $\overline{G}$ on the circle:

\begin{proposition}
\label{prop:rot:G}

Let $G \in \mathcal{G}$. Let $\rho$ be an action of $\overline{G}$ on the circle without global fixed points, and let $\rho_{st}$ be the standard action, i.e., the one induced by the unique (up to conjugacy) proximal action of the quotient $G$. Then, for all $g \in \overline{G}$ it holds that
$$ \rot(\rho(g)) \in \mathbf{Q} / \mathbf{Z} \iff \rot(\rho_{st}(g)) \in \mathbf{Q} / \mathbf{Z}.$$
\end{proposition}

\begin{proof}
Let $\rho_m$ be the minimalization of $\rho$. Since every non-trivial quotient of $\overline{G}$ surjects onto $G$, $\rho_m(\overline{G})$ is not abelian, so we can produce a proximalization $\rho_p$. By Theorem \ref{thm:prox} the centralizer of $\rho_p(\overline{G})$ is trivial, therefore the whole center of $\overline{G}$ lies in the kernel of $\rho_p$. It follows that $\rho_p$ induces a proximal action of $G$, and so by the definition of $\mathcal{G}$, we deduce that $\rho_p$ is conjugate to $\rho_{st}$. We conclude by Proposition \ref{prop:rot:general} and the fact that rotation numbers are conjugacy-invariant \cite[Section 5.1]{ghys:book}.
\end{proof}

\subsection{Examples}

Here we provide some examples of groups in $\mathcal{G}$. We start with the most well-known example:

\begin{example}[Thompson]
\label{ex:T}

Our first example is Thompson's group $T$, which is the group of orientation-preserving piecewise linear homeomorphisms of the circle $\mathbf{R}/\mathbf{Z}$ stabilizing $\mathbf{Z}\left[\frac{1}{2}\right]$, with slopes in $2^{\mathbf{Z}}$ and breakpoints in $\mathbf{Z}\left[\frac{1}{2}\right]$. It is well-known that $T$ lies in $\mathcal{G}$ and that it is of type $F_\infty$; all necessary facts can be found in \cite{cfp, GhysSergiescu, Brown}. For the uniqueness of the action on the circle, this is only stated explicitly in \cite{GhysSergiescu} in a weaker form (uniqueness among $C^2$ actions), however it can be deduced from their cohomological computations, or from the rest of the discussion in this section.
\end{example}

For the remainder of the examples, simplicity and torsion can be found in the literature, and perfection of the lift is easy to check and can sometimes be reduced to the case of $T$. The uniqueness of the proximal action is less well-studied, but can be established via the following general criterion:

\begin{theorem}[{\cite[Proposition 6.9]{binate}}, {\cite[Corollary 4.2]{spectrum}}]\label{thrm:unique_prox}
Let $G \le \Homeo^+(\mathbf{S}^1)$ be either piecewise linear or piecewise projective. Suppose that there exists a dense orbit $Y \subset \mathbf{S}^1$ such that the action of $G$ on positively oriented triples in $Y$ is transitive. Then the given action of $G$ on $\mathbf{S}^1$ is proximal, and it is the unique proximal action of $G$ on the circle, up to conjugacy.
\end{theorem}

This next example provides an infinite family of groups in $\mathcal{G}$, which will be used in the proof of Corollary~\ref{cor:main}:

\begin{example}[Stein--Thompson]
\label{ex:T23}

For $n_1,\dots,n_k\in\mathbf{N}$ ($n_i\ge 2$) let $\lambda := n_1\cdots n_k$, let $Y := \mathbf{Z}\left[\frac{1}{\lambda}\right]/\mathbf{Z}$ and let $P := \langle n_1, \ldots, n_k \rangle \leq \mathbf{R}^{\times}$. The \emph{Stein--Thompson group} $T_{n_1,\dots,n_k}$ is the group of orientation-preserving piecewise linear homeomorphisms of the circle $\mathbf{R}/\mathbf{Z}$, stabilizing $Y$, with slopes in $P$ and breakpoints in $Y$; see \cite{stein}. When $k=1$ and $n_1=2$, we recover the classical Thompson's group $T$.

We claim that any Stein--Thompson group of the form $G := T_{2,3,n_3,\dots,n_k}$, is in $\mathcal{G}$, for example $T_{2,3}$. All Stein--Thompson groups clearly have torsion, and these particular Stein--Thompson groups are all simple \cite{guelman}. Since $\overline{T} \leq \overline{G}$ we deduce that the group of translations is contained in the commutator subgroup of $\overline{G}$. Together with simplicity of $G$, this shows that $\overline{G}$ is perfect. 

Finally, we claim that Theorem~\ref{thrm:unique_prox} applies, with $Y$ as above; in fact, we claim that the action is transitive on positively ordered $p$-tuples for all $p \geq 1$. Let $x_1, \ldots, x_p$ and $y_1, \ldots, y_p$ be two positively ordered $p$-tuples in $Y$. Then there exists some $q \geq 1$ such that the representatives of all $x_i$ and $y_i$ in $[0, 1)$ are integer multiples of $\lambda^{-q}$. Setting $\mathbf{x} = \mathbf{y} = \{ 0, \ldots, 1 - \lambda^{-q} \} \mod \mathbf{Z}$ includes both $p$-tuples inside larger tuples $\mathbf{x}, \mathbf{y}$ with the additional property that the difference between two consecutive points lies in $P$. Now, adding the midpoint of two consecutive points to $\mathbf{x}$ or $\mathbf{y}$ keeps the elements of the tuple inside $Y$, and does not change the property that the difference between two consecutive points in $\mathbf{x}$ or $\mathbf{y}$ lies in $P$. Therefore we can add midpoints inductively to $\mathbf{x}$ or $\mathbf{y}$ to ensure that for each $i = 1, \ldots, p$, the number of points in the arc $\mathbf{x} \cap (x_{i-1}, x_i)$ is the same as the number of points in the arc $\mathbf{y} \cap (y_{i-1}, y_i)$ (where we abuse the notation and interpret $x_0, y_0$ as $x_p, y_p$). This ensures that the piecewise linear homeomorphism of $\mathbf{S}^1$ sending $\mathbf{x}$ to $\mathbf{y}$ and interpolating linearly also sends $x_i$ to $y_i$ for each $i$. Since $\mathbf{x}$ and $\mathbf{y}$ are contained in $Y$, and the difference between two consecutive points is in the multiplicative group $P$, the corresponding homeomorphism of the circle is an element of $G$ that sends $x_i$ to $y_i$ for each $i$, as desired.

This shows that $T_{2,3,n_3,\dots,n_k}$ is in $\mathcal{G}$. Also note that all Stein--Thompson groups are of type $F_\infty$ \cite{stein}, and that any $T_{2,3,n_3,\dots,n_k}$ has elements with irrational rotation number \cite{liousse}.
\end{example}

We end with two more exotic examples.

\begin{example}[Golden Ratio Thompson]
Let $\tau := \frac{\sqrt{5} - 1}{2}$ be the small golden ratio. The Golden Ratio Thompson's group $T_{\tau}$ is a circle version of Cleary's Golden Ratio Thompson's group $F_{\tau}$ \cite{Ftau1, Ftau2}, introduced in \cite{Ttau}. The commutator subgroup $T_\tau'$ has index $2$ and is simple \cite[Thoerem 3.2]{Ttau}. Moreover, the action of $T_\tau'$ on $\mathbf{Z}[\tau]$ is transitive on positively oriented $n$-tuples for all $n$ \cite[Lemmas 5.5 and 5.7]{spectrum}, so Theorem~\ref{thrm:unique_prox} applies to $T_\tau'$. Moreover, $T_\tau'$ is of type $F_\infty$ \cite[Proposition 5.8]{spectrum}, and the lift $\overline{T_\tau'}$ is perfect \cite[Lemma 5.13]{spectrum}. Thus, $T_\tau'\in\mathcal{G}$. Finally, $T_\tau'$ clearly has elements with irrational rotation number, for instance the rotation by $2 \tau$ (see \cite{spectrum} for more detail).
\end{example}

\begin{example}[Piecewise projective]
The group $S$ introduced in \cite{S} is simple, has torsion, and is of type $F_\infty$ \cite{SFinfty}. It is a group of piecewise projective homeomorphisms of the circle $\mathbf{R}\cup\{\infty\}$, which is the real projective line. The action of $S$ preserves $\mathbf{Q}\cup \{\infty\}$, and it acts transitively on positively oriented triples of this set. (Indeed, the copy of $T$ in $S$ already acts transitively on such triples). Thus Theorem~\ref{thrm:unique_prox} applies to $S$. Moreover, $\overline{S}$ is perfect with a similar argument as above: its derived subgroup contains $\overline{T}$, and therefore all of the integer translations. Thus $S\in\mathcal{G}$. We do not know whether $S$ has elements with irrational rotation number.
\end{example}

\section{Monstrousness and the proof of Theorem \ref{thm:main}}

Now we can prove our main result.

\begin{proof}[Proof of Theorem \ref{thm:main}]
We recall our hypothesis: $G_1,G_2\in \mathcal{G}$ are groups such that $G_2$ contains an element with an irrational rotation number. For example, for the special case in Theorem~\ref{thm:main23} one can use $G_1=T$ (Example \ref{ex:T}) and $G_2=T_{2,3}$ (Example \ref{ex:T23}). We denote as usual by $\overline{G}_i$ the lift of the standard action of $G_i$, we let $z$ denote the generator of the center of $G_1$, and let $g\in \overline{G}_2$ denote a lift of an element of $G_2$ with irrational rotation number. \\ 

{\bf Ad 1}: 
The left orderability of $M$ follows from the fact that an amalgamated product of two left orderable groups along a cyclic subgroup is left orderable \cite{bludov:glass:cyclic} (see \cite{bludov:glass} for a more general result).

Being an amalgamated product of two perfect groups, $M$ is perfect, therefore it has no type 1 action. Now suppose by contradiction that $M$ admits a fixed point-free type $2$ action $\rho : M \to \Homeo^+(\mathbf{R})$, which we may assume to be minimal by Proposition \ref{prop:minimalization}. This implies the existence of a fixed point-free element $t \in \Homeo^+(\mathbf{R})$ that commutes with $\rho(M)$.
Up to conjugating our action by a suitable homeomorphism, we may assume for the rest of the proof that $x\cdot t=x+1$ for each $x\in \mathbf{R}$ (this follows for example from H{\"o}lder's theorem \cite[Theorem 3.1.2]{book}).
Also, we declare $ H_1=\rho(\overline{G}_1)$ and $H_2=\rho(\overline{G}_2)$.

We first claim that neither $H_1$ nor $H_2$ have global fixed points in $\mathbf{R}$. Indeed, suppose by contradiction that $H_1$ fixes a point $x$ (the proof for $H_2$ is identical). Since $\rho(M)$ has no global fixed points, $H_2$ does not fix $x$. Thus there is an $H_2$-invariant open interval $J$ containing $x$ on which the restriction of the action of $H_2$ is fixed point-free. Since $\overline{G}_2$ does not have proper non-trivial torsion-free quotients by Lemma \ref{lem:TbarQuotients}, this action on $J$ is faithful.

By Lemma \ref{lem:lineactions}, the action of $H_2$ on $J\cong\mathbf{R}$ is semiconjugate to the standard action. In particular, $\rho(g)$ does not fix a point in $J$: otherwise the rotation number of the image of $g$ in $G_2$ would be $0$. But looking at the action of all of $M$, this contradicts the fact that $\rho(g) = \rho(z) \in H_1$ fixes $x$. This proves our claim.
Moreover, it follows from Lemma \ref{lem:lineactions} that $H_1$ and $H_2$ are semiconjugate to their standard actions on the real line.

Now the quotient $\mathbf{R} \to \mathbf{R}/\langle t \rangle \cong \mathbf{S}^1$ induces a homomorphism $\phi:M\to \Homeo^+(\mathbf{S}^1)$ with kernel $\langle t\rangle \cap M$. The restriction of $\phi$ to each $\overline{G}_i$ is fixed point-free on $\mathbf{S}^1$. Indeed, if it fixed a point then this fixed point would lift to a discrete orbit, which would have to be fixed pointwise because $\overline{G}_i$ is perfect (and hence non-indicable). But we have already shown that the restriction of $\rho$ to $\overline{G}_i$ has no global fixed points. 

We can therefore use Proposition \ref{prop:rot:G} on $\overline{G}_1$ to obtain $\rot(\phi(z)) \in \mathbf{Q}/\mathbf{Z}$, and then on $\overline{G}_2$ to obtain $\rot(\phi(g)) \notin \mathbf{Q}/\mathbf{Z}$. But $\phi$ is an action of $M$, so $\phi(g) = \phi(z)$, a contradiction. \\

{\bf Ad 2}: It is a standard fact that an amalgamation of two type $F_\infty$ groups along a type $F_\infty$ subgroup is of type $F_\infty$; indeed, the amalgamation acts cocompactly on a tree with type $F_{\infty}$ vertex and edge stabilizers, and hence is itself $F_\infty$ for example by Propositions~1.1 and~3.1 of \cite{Brown}. The same holds for other finiteness properties, such as finite presentability. Therefore, part $2.$ of the Theorem follows. \\

{\bf Ad 3}: We use the following fact \cite{bc:amalgam1} (see also \cite{bc:amalgam2}): an amalgamated product $A *_C B$ has an infinite-dimensional space of homogeneous quasimorphisms whenever $C \neq B$ and $C$ has at least $3$ double cosets in $A$. This automatically holds if $C$ is not maximal in $A$, in particular it applies to $M$, which therefore has an infinite-dimensional space of homogeneous quasimorphisms. Being moreover perfect, all such quasimorphisms represent non-trivial classes in bounded cohomology, and witness that $M$ is not uniformly perfect \cite{scl}. \\

{\bf Ad 4}: The fact that $M$ admits a proximal action follows from the fact that the quotient of $M$ by the normal closure of $\overline{G}_2$ is $G_1$, which acts proximally on the circle.
 We now show that this action of $M$ is unique (up to conjugacy).

 Let $\rho$ be a proximal action of $M$ on the circle. 
 First we claim that $\rho(g)$ fixes a point.
 Assume otherwise. In this case, we obtain that both actions  $\rho(\overline{G}_1),\rho(\overline{G}_2)$ are fixed point-free, since the element $\rho(g)=\rho(z)$ acts without fixed points on the circle and is contained in both actions. Since $\overline{G}_1,\overline{G}_2$ do not admit non-trivial finite quotients by Lemma \ref{lem:TbarQuotients}, it further follows that $\rho(\overline{G}_1),\rho(\overline{G}_2)$ do not admit a finite orbit. 
This means that they both admit proximalizations, using Theorems \ref{thm:circleAction} and \ref{thm:prox}.
Now using the same argument as in the proof of part $1$, we obtain the contradiction that $\rho(g)=\rho(z)$ has both a rational and an irrational rotation number. This proves our claim that $\rho(g)$ fixes a point in the circle.

Next, we claim that $\rho(\overline{G}_2)$ fixes a point in the circle.
If not, again it also cannot admit a finite orbit, so by Theorems~\ref{thm:circleAction} and~\ref{thm:prox} it must admit a proximalization $\rho'$. But since this is conjugate to the standard action (by our assumption on groups in $\mathcal{G}$), this means that $\rho'(g)$ has an irrational rotation number. But this is impossible, since $\rho(g),\rho'(g)$ are semiconjugate, and $\rho(g)$ admits a fixed point, so it has rotation number $0$. This proves that $\rho(\overline{G}_2)$ fixes a point in the circle.

The complement of $Fix(\rho(\overline{G}_2))$ is a union of open intervals, each of which are individually $\rho(\overline{G}_2)$-invariant. Since $\overline{G}_2$ does not admit a nontrivial left orderable quotient by Lemma \ref{lem:TbarQuotients}, the action of $\rho(\overline{G}_2)$ on each these intervals is faithful. Thanks to Lemma \ref{lem:lineactions}, each of these is semiconjugate to the standard action of $\overline{G}_2$ on $\mathbf{R}$. Since in the standard action, $g$ does not fix a point, this implies that $Fix(\rho(\overline{G}_2))=Fix(\rho(g))$.

Now let $X=Fix(\rho(\overline{G}_2))=Fix(\rho(g)) \subset \mathbf{S}^1$. Clearly, $X$ is a closed subset of the circle, and $X$ is $\rho(\overline{G}_1)$-invariant, since $Fix(\rho(g))=Fix(\rho(z))$ and $\rho(z)$ is central in $\rho(\overline{G}_1)$. Hence $X$ is a closed non-empty $\rho(M)$-invariant set. By the minimality of $\rho(M)$, we conclude that $X=\mathbf{S}^1$. It follows that the element $\rho(g)=\rho(z)$ fixes the circle pointwise, and hence lies in the kernel $\rho:M\to \Homeo^+(\mathbf{S}^1)$. Therefore $\rho(M)$ descends to a proximal action of $G_1$, which is unique by our assumption.
\end{proof}

We conclude with the main corollary.

\begin{proof}[Proof of Corollary \ref{cor:main}]

Let $G_1,G_2,M$ be as in the statement of Theorem \ref{thm:main}, and moreover assume that they are finitely presented (or type $F_\infty$).
Part $4$ of Theorem \ref{thm:main} tells us that $G_1$ is an invariant for the isomorphism type of $M$. The class $\mathcal{G}$ contains the finitely presented (even type $F_\infty$) groups $T_{2,3,n_3,\dots,n_k}$ (Example \ref{ex:T23}), which are pairwise non-isomorphic \cite{liousse}. We conclude that finitely presented (even type $F_\infty$) $M$ can also take on infinitely many isomorphism types.
\end{proof}

\footnotesize

\bibliographystyle{abbrv}
\bibliography{references}

\normalsize

\vspace{0.5cm}

\noindent{\textsc{Department of Mathematics, ETH Z\"urich, Switzerland}}

\noindent{\textit{E-mail address:} \texttt{francesco.fournier@math.ethz.ch}} \\

\noindent{\textsc{Department of Mathematics,
University of \Hawaii at \Manoa.}}

\noindent{\textit{E-mail address:} \texttt{lodha@hawaii.edu}} \\

\noindent{\textsc{Department of Mathematics and Statistics, University at Albany (SUNY)}}

\noindent{\textit{E-mail address:} \texttt{mzaremsky@albany.edu}}

\end{document}